\documentclass[12pt]{article}
\usepackage{amsmath,amsfonts,amssymb,amsthm}
\usepackage{bbm}
\usepackage{setspace}

\usepackage{geometry}
\usepackage{enumerate}

\usepackage{comment} 
\usepackage[numbers]{natbib}

\usepackage[affil-it]{authblk} 

\usepackage[colorlinks=false,linktoc=all,bookmarksopen=true,linktocpage,pdftex]{hyperref}
\usepackage[pdftex]{bookmark}
\hypersetup{
	colorlinks=true,       
	linkcolor=blue,        
	citecolor=blue,        
	filecolor=magenta,     
	urlcolor=blue         
}

\usepackage{tabularx} 
\usepackage{mathtools}
\usepackage{cleveref}
\usepackage{enumitem}

\theoremstyle{plain}
\newtheorem{thm}{Theorem}[section]
\newtheorem{lem}[thm]{Lemma}
\newtheorem{lemma}[thm]{Lemma}
\newtheorem{prop}[thm]{Proposition}

\newtheorem{rem}[thm]{Remark}

\theoremstyle{definition}

\newtheorem{definition}[thm]{Definition}

\newcommand{\refT}[1]{Theorem~\ref{#1}}

\newcommand{\refL}[1]{Lemma~\ref{#1}}

\newcommand{\refPp}[1]{Proposition~\ref{#1}}

\newcommand\lrset[1]{\ensuremath{\left\{#1\right\}}}

\newcommand\lrpar[1]{\left(#1\right)}

\newcommand\lrcpar[1]{\left\{#1\right\}}
\newcommand\abs[1]{|#1|}

\newcommand\ceil[1]{\lceil#1\rceil}

\newcommand\floor[1]{\lfloor#1\rfloor}

\newcommand{\norm}[1]{\left\lvert#1\right\rvert}
\newcommand{\hnorm}[1]{\left\lVert#1\right\rVert}
\newcommand{\E}{\mathbb{E}}
\newcommand{\R}{\mathbb{R}}

\newcommand{\sparse}{{\rm Sparse}}
\newcommand{\sfe}{{\mathbb{S}^{n-1}}}
\newcommand{\supp}{{\rm Supp}}
\newcommand{\comp}{{\rm Comp}}
\newcommand{\incomp}{{\rm Incomp}}
\newcommand{\dist}{{\rm dist}}
\newcommand{\colspan}{{\rm colspan}}

\newcommand{\cE}{\mathcal{E}}
\newcommand{\cI}{\mathcal{I}}
\newcommand{\cJ}{\mathcal{J}}

\newcommand{\cN}{\mathcal{N}}
\newcommand{\bS}{\mathbb{S}}
\newcommand{\cS}{\mathcal{S}}
\newcommand{\cT}{\mathcal{T}}

\newcommand{\bp}{\mathbb{P}}
\newcommand{\ZZ}{\mathbb{Z}}
\newcommand{\RR}{\mathbb{R}}
\newcommand{\eps}{\varepsilon}

\title{On the Smallest Singular Value of Log-Concave Random Matrices}

\author{Galyna V. Livshyts}
\affil{Georgia Institute of Technology, School of Mathematics, Atlanta, GA, USA 30332 \texttt{glivshyts6@math.gatech.edu}}

\author{Manuel Fernandez V}
\affil{Georgia Institute of Technology, School of Mathematics, Atlanta, GA, USA 30332 
\texttt{mfernandez39@gatech.edu}}

\author{Stephanie Mui}
\affil{Georgia Institute of Technology, School of Mathematics, Atlanta, GA, USA 30332 
\texttt{smui3@gatech.edu}}

\date{}

\begin{document}

\maketitle 

\begin{abstract}

Let $A$ be an $N\times n$ random matrix whose entries are coordinates of an isotropic log-concave random vector in $\R^{Nn}$. We prove sharp lower tail estimates for the smallest singular value of $A$ in the following cases: (1) when $N=n$ and $A$ is drawn from an unconditional distribution, with no independence assumption; (2) when the columns of $A$ are independent and $N\geq n$; (3) when $A$ is sufficiently tall, that is $N\geq (1+\lambda)n$ for any positive constant $\lambda$. 
\end{abstract}

\section{Introduction}

If $A$ is an $N\times n$ random matrix, with $N\geq n,$ its singular values are $\sigma_1(A)\geq ...\geq \sigma_n(A)$. In particular,
$$\sigma_1(A)=\sup_{x\in\sfe} |Ax|,\,\,\,\,\,\,\,\sigma_n(A)=\inf_{x\in\sfe} |Ax|,$$
where $|\cdot|$ stands for the Euclidean norm. One of the key questions in random matrix theory is bounding the condition number of a random matrix, given by $\frac{\sigma_1(A)}{\sigma_n(A)}$, from above. This question is directly related to the speed of algorithms solving the system of equations $Ax=b,$ see e.g. \cite{RudVer-icm}, \cite{taovu-annals}. In particular, non-asymptotic random matrix theory focuses on estimates which hold in any large but fixed dimension with high probability (and the probability estimate improves as the dimension grows). Both upper estimates on the largest singular value and the lower estimates on the smallest singular value have been studied extensively in the non-asymptotic setting, especially in the case of matrices with a lot of independencies, including the class of matrices with all independent entries \cite{LPRT, RebTikh, RV08, RudVer-general, szarek, Tikh, TikhErd, Tikh1, Tikh-nomoments, tatarko, taovu-tall, taovu-annals, taovu, taovu-1, Versh-tall, wei}. However, without independence, little is understood about any of these questions. 

An interesting motif in High-Dimensional Probability is discovering to what extent \emph{independence} may be replaced by \emph{convexity}, while preserving various high-dimensional phenomena. One of the most notable works in this direction is the Central limit theorem for convex sets due to Klartag \cite{Kl, Kl2}, which resolved a conjecture from \cite{ABP}. More generally, concentration properties of random vectors with convexity properties are similar to those of random vectors with independencies. Two major conjectures on this subject \cite{bour1, bour, ABP, BobkovKoldobsky} were recently settled by Klartag, Lehec \cite{KL-slicing, KL-thinshell} (see also related works \cite{ChenKLS, Ronen, GueMil, KLS, K22, Kl_LK, L24, LV}). 

In this paper, we study random matrices with jointly log-concave distribution of entries. Recall that a random vector $X$ in $\R^d$ is called log-concave if it is distributed according to a density supported on a convex subset of an affine subspace, whose logarithm is a concave function. A random vector is called isotropic if $\E X=0$ and its covariance matrix ${\rm Cov}(X)=(\E(X_i X_j))$ is identity. Every centered random vector whose density has non-empty support can be made isotropic via a linear transformation. We consider an $N\times n$ random matrix $A$ whose entries together form a log-concave isotropic random vector in $\R^{Nn}$. Little is known about this ensemble of random matrices. One might hope that similar to how log-concave random vectors exhibit strong concentration properties of random vectors with independencies, this ensemble of log-concave random matrices would also exhibit phenomena reminiscent of random matrices with independent entries. To date, the only results in this direction are due to Mendelson, Pajor, Tomczak-Jaegermann \cite{Mend-et-al}, where the very tall case with independent columns was considered, and Adamczak, Litvak, Pajor, Tomczak-Jaegermann \cite{A10}, which states that with high probability, the largest singular value of nearly-square random matrices with jointly log-concave isotropic distribution with i.i.d. columns is of order $C\sqrt{n}.$ Until now, nothing was known neither of the case of the norm of the log-concave random matrix whose columns are dependent, nor of the smallest singular value of log-concave random matrices. We present sharp (up to log factors) bounds on the norm of such random matrices in a separate forthcoming paper \cite{FLMV-largest}. In the present paper, we focus on the smallest singular value. We show the optimal bound, both in terms of the value and in terms of probability, in the following cases:

\begin{itemize}
\item $N\geq (1+\lambda)n$, where $\lambda$ is any positive constant, and $A$ is log-concave and isotropic;
\item The distribution of $A$ is log-concave, isotropic and unconditional, and $N=n$;
\item $N\geq n$ and the columns of $A$ are independent, isotropic and log-concave.
\end{itemize}

Our first result is an optimal lower tail estimate for the smallest singular value when $A$ is a sufficiently tall matrix:
\begin{thm}\label{thm:tall}
 Let $\lambda > 0$. Let $A \in \R^{N \times n}$ be an isotropic log-concave matrix (that is, it has an isotropic log-concave distribution when considered as a vector in $\R^{nN}$), with $N \ge (1+\lambda)n$. Then there exists $c > 0$, depending only on $\lambda$, such that
\begin{equation}\label{eq:tall}
    \bp\lrpar{\sigma_n(A) \le c\sqrt{N}} \le e^{-c'N},
\end{equation}
where $c, c'>0$ are absolute constants.
\end{thm}
Our next result is the tight lower tail estimate for the smallest singular value when $A$ is square and unconditional (i.e. the density of $A$, when viewed as a vector, is symmetric over all coordinate hyperplanes):
\begin{thm}\label{thm:unconditional}
     Let $\eps > 0$. Let $A \in \R^{n \times n}$ be an unconditional isotropic log-concave matrix. Then 
    \begin{equation}\label{eq:unconditional}
    \bp\lrpar{\sigma_n(A) \le \frac{\eps}{\sqrt{n}}} \le C\eps + e^{-cn},
    \end{equation}
where $C,c > 0$ are absolute constants.
\end{thm}
Our last result is the tight lower tail estimate for the smallest singular value when $A$ has independent columns:
\begin{thm}\label{thm:indep-columns}
    Let $A \in \R^{N \times n}$ be an isotropic log-concave matrix with independent columns. Then
    \begin{equation}\label{eq:indep-columns}
    \bp\lrpar{\sigma_n(A) \le \eps(\sqrt{N+1} - \sqrt{n})} \le (C\eps)^{N-n+1} + e^{-cN},
    \end{equation}
    where $C,c > 0$ are absolute constants. 
\end{thm}

We remark that under the additional assumption of independent columns, Theorem \ref{thm:indep-columns} generalizes both Theorems \ref{thm:tall} and \ref{thm:unconditional}, and that this small ball behavior for the smallest singular value is known to be exhibited by matrices with independent entries, under certain assumptions on the moments of the entries. Namely, in their seminal work, Rudelson and Vershynin \cite{RV08} showed this behavior for sub-Gaussian i.i.d. random matrices with $N=n,$ and later for $N\geq n$ in \cite{RudVer-general}. In the case of $N=n$, this was extended to i.i.d. matrices with uniformly anti-concentrated entries with bounded second moments by Rebrova, Tikhomirov \cite{RebTikh}, the i.i.d. assumption was relaxed by Livshyts \cite{L21}, and this result was obtained assuming independent uniformly anti-concentrated entries with bounded expected Hilbert-Schmidt norm by Livshyts, Tikhomirov, Vershynin \cite{LTV}. The general $N\geq n$ case was studied by Livshyts \cite{L21} and later by Dabagia, Fernandez \cite{Fern-Dab}, utilizing the work of Fernandez \cite{Fern-distance}.

On a high level, our proof uses the compressible-incompressible decomposition of the sphere, which was pioneered by Rudelson and Vershynin and subsequently used in hundreds of follow-up works. At the core of our argument, we use the recent positive resolution of the celebrated Bourgain's slicing problem \cite{bour1, bour} by Klartag and Lehec \cite{KL24} (see also Bizeul \cite{B25} for an alternative argument). We use it for the point-wise small ball estimate for the length of a log-concave random vector, coupled with a net argument. In order to avoid using the estimate on the norm of the log-concave matrix, which we have with a logarithmic error (as already mentioned, we prove it in a separate paper \cite{FLMV-largest}), we use a net argument based on random rounding which was developed in \cite{KlLi}, \cite{L21} (following ideas from \cite{KA}) and later used also e.g. in \cite{Jain-Silwal, Fern-distance, Fern-Dab, LTV}. 

Significant challenges arise in the implementation of this scheme. For instance, in the incompressible case, we need small ball estimates not for the isotropic log-concave vector (which is our matrix), but for a sub-vector (a column of our matrix) conditioned on the realization of the random normal (orthogonal to the span of the other columns of the matrix). Conditioning preserves log-concavity but loses isotropicity, and therefore we cannot use the slicing estimate \cite{B25, KL24} ready-made. We manage to address this challenge in the tall case, the square unconditional case, as well as in the independent columns case. However, in general, this issue appears to be significant, and not just a minor technicality, as can be seen in some examples. In view of that, we are not convinced that the answer should remain the same in the general log-concave isotropic case; this question remains an interesting open problem.

In Section 2, we present some preliminaries. In Section 3, we outline the proof in the tall case. In Section 4, we outline the proof in the square unconditional case. In Section 5, we discuss the independent columns case for the general aspect ratio.  

\textbf{Acknowledgment.} The authors are grateful to Santosh Vempala and Konstantin Tikhomirov for helpful discussions and interest. The first named author is grateful to ARCS fellowship. The second named author is supported by NSF-BSF DMS grant 2247834. The third named author is supported by NSF DMS grant 2402038.

\section{Preliminaries}

We work in $\R^n$ equipped with the Euclidean norm which we denote by $|\cdot|$. The scalar product is denoted by $\langle \cdot, \cdot\rangle.$ The Hilbert-Schmidt norm of a matrix $A$ is denoted as $\hnorm{A}_{HS}$ and is defined as 
$\hnorm{A}_{HS} := \sqrt{\sum_{i,j} a_{ij}^2}$. 
Throughout this paper, unless stated otherwise, the letters $c,C,c_1
, C_1$, etc.
denote positive absolute constants, that are not necessarily the same in
different appearances. Finally, $C$ and $c$ and may be used multiple times in the same context, such as in a chain of inequalities, with different values at every appearance, in order to combine or simplify terms in a complicated expression. 
\par 
Recall the notions of sparse, compressible and incompressible vectors originating from \cite{RV08}:
\begin{equation}\label{eq:comp-decomp}
\begin{split}
    \sparse(\delta n) &:= \lrset{x \in \R^n: \supp(x) \le \delta n}, \\
    \comp(\delta,\rho,n) &:= \left\{ x \in \bS^{n-1}: \dist(x, \sparse(\delta,n)) \le \rho\right\}, \\
    \incomp(\delta,\rho,n) &:= \bS^{n-1}\setminus \comp(\delta,\rho,n). 
\end{split}
\end{equation}
We will take the convention of dropping $n$ from the above notation when the ambient dimension of the sets are clear from the context. Given a matrix $A$ and a subset of column indices $J$ we write $A_J$ to denote the submatrix consisting of all columns of $A$ indexed by $J$ and write $H_J$ to denote the column span of $A_J$ (the $A$ that defines $H$ will always be clear from the context). If $J = \{j\}$ we may write $A_j$ instead of $A_J$. Give a matrix $A$ we use ${\colspan}(A)$ to denote the column span of $A$.
\par We now state a number of results that we will use throughout the paper. The first is the small ball estimate for isotropic log-concave vectors, a breakthrough result due to Klartag and Lehec \cite{KL24}, and Bizeul \cite{B25} (for concrete reference, see Corollary 3 in \cite{B25}).

\begin{thm}[Klartag, Lehec; Bizeul]\label{thm:Bizeul}
    Let $X$ be an isotropic log-concave vector in $\R^n$. Then for any $\eps > 0$ and any $y \in \R^n$, 
    \begin{equation}\label{eq:Bizeul}
    \bp(\norm{X-y} \le \eps\sqrt{n}) \le (C\eps)^n.
    \end{equation}
\end{thm}
Next result is a seminal large deviation inequality for the norm of an isotropic log-concave vector, due to Paouris \cite{P06}. 
\begin{thm}[Paouris]\label{thm:Paouris}
If $X$ is an isotropic log-concave random vector in $\R^n$, then for every $t \ge 1$, 
\begin{equation}\label{eq:Bizeul}
\bp(|X| \ge Ct\sqrt{n}) \le \exp(-t\sqrt{n}),
\end{equation}\label{eq:bizeul}
where $C > 0$ is an absolute constant.
\end{thm}
Finally, we recall the following result about lattice-like $\eps$-nets, see Lemma 5.1 in Klartag, Livshyts \cite{KlLi} (which was later used as Lemma 3.5 in Livshyts \cite{L21} for a more involved construction, but in this paper only require this result).
\begin{lem}[a net via random rounding]
\label{lem:livshyts}\label{Livshyts_net_sv}\label{lem:net-full} Let $\varepsilon\in(0,0.05)$. Then
there exists a net $\mathcal{N}\subset\frac{3}{2}B_{2}^{n}\backslash\frac{1}{2}B_{2}^{n}$
with $\#\mathcal{N}\leq\left(\frac{C}{\varepsilon}\right)^{n-1}$
such that for any (deterministic) matrix $A$ and every $x\in S^{n-1}$,
there exists $y\in\mathcal{N}$ such that
\begin{equation}\label{eq:livshyts}
\norm{A(y-x)}\leq\frac{\varepsilon}{\sqrt{n}}\hnorm{A}_{HS},
\end{equation}
 where $C > 0$ is an absolute constant.
\end{lem}

\section{Lower tail estimates in the tall case}

In order to prove \refT{thm:tall}, we first prove a lower tail estimate on $\inf_{x \in \comp(\delta,\rho)}|Ax|$.
\begin{thm}\label{thm:comp}
    Let $\delta \in (0,1)$. Then there exists $c,\rho \in (0,1)$ depending on $\delta$ such that the following holds: Let $A \in \R^{N \times n}$ be an isotropic log-concave matrix. Then    
    \begin{equation}\label{eq:comp}
    \bp\lrpar{\inf_{x \in \comp(\delta,\rho)} \norm{Ax} \le c\sqrt{N}} \le e^{-c\sqrt{Nn}}.
    \end{equation}
\end{thm}
As a corollary of Theorem \ref{thm:comp}, one can deduce \refT{thm:tall} by reducing to the compressible case.
\begin{proof}[Proof of \refT{thm:tall}]
     We will address the $N/n \in (1+\lambda, 2)$ case as a remark after the proof of \refT{thm:comp}.
     Suppose then that $N/n \ge 2$. Take $\rho,c$ from Theorem \ref{thm:comp} with $\delta = 1/2$. We define $\tilde{A} \in \R^{N \times (n \floor{N/n})}$ to be a random matrix obtained by concatenating together $s := \floor{N/n}$ independent copies of $A$ along their rows. From the definition of $\tilde{A}$ it follows that $\tilde{A}_{\{1,\cdots,n\}}$ has the same law as $A$, and $\tilde{A}$ is also an isotropic log-concave matrix. Therefore, since
    $$\bS^{n-1} \times \{0\}^{n\cdot(s - 1)} \subset \comp(1/2,\rho,sn),$$ we have that 
    \[
    \bp\lrpar{\inf_{x \in \bS^{n-1}} \norm{Ax} \le c\sqrt{N}} \le \bp\lrpar{\inf_{y \in \comp(1/2,\rho,sn)} \lVert \tilde{A}x\rVert_2 \le c\sqrt{N}} \le e^{-cN},
    \]
with the last inequality following from \eqref{eq:comp} of Theorem \ref{thm:comp} and the fact that $(N/n)/\floor{N/n} \le 3/2$ whenever $N/n \ge 2$. 
\end{proof}
\begin{rem}
    We note that the reduction to the compressible case is necessary for the tight probability estimate. If one instead tried to work directly with $\inf_{x \in \bS^{n-1}} \norm{Ax}$ and apply the net argument, one would obtain a bound of the form
    \[
    \bp\lrpar{\inf_{x\in\bS^{n-1}} \norm{Ax} \le c\sqrt{N}} \le e^{-c\sqrt{Nn}}.
    \]
    This is because the net argument requires us to bound the probability of two events:
    \begin{enumerate}
    \item $\{\inf_{y \in \cN} \norm{Ay} \le c\sqrt{N}\}$, for an appropriate net $\cN$ and absolute constant $c > 0$. 
    \item $\{\hnorm{A}_{HS} \ge C\sqrt{Nn}\}$, for an appropriate absolute constant $C > 0$.
    \end{enumerate}
    The first event holds with probability at most $e^{-cN}$ while the second event holds with probability at most $e^{-c\sqrt{Nn}}$, as guaranteed by Paouris's inequality. 
    By working with $\tilde{A}$ instead of $A$, the bound on the probability of the first event is larger (the size of $\cN$ is now greater) but is still at most $e^{-cN}$ (with a potentially smaller $c$). 
    On the other hand, the second event now occurs with probability at most $e^{-cN}$, once again by Paouris's inequality.  
\end{rem}
We now turn our attention to the proof of Theorem \ref{thm:comp}. The proof relies on an epsilon net argument that employs a sparse lattice-like net coupled with the Klartag-Lehec and Bizeul's slicing estimate \refT{thm:Bizeul}, and Paouris's inequality (Theorem \ref{thm:Paouris}). 

First, we recall the notion of random rounding (see e.g. \cite{KA, KlLi, L24}.) 
For a vector $y \in \R^n$ and parameter $\eps > 0$ we define a random vector $\eta_y \in \frac{\eps}{\sqrt{n}}\ZZ \times \frac{\eps}{\sqrt{n}}\ZZ \times \cdots \times \frac{\eps}{\sqrt{n}}\ZZ$ with independent coordinates such that $\hnorm{y-\eta_y}_\infty \le \eps/\sqrt{n}$ almost surely and $\E \eta_y = y$. Specifically, writing $y_i = (\eps/\sqrt{n})(k_i+p_i)$, where $k_i \in \ZZ, p_i \in [0,1)$, we have

\begin{equation}\label{eq:round-def}
    (\eta_y)_i = 
    \begin{cases}
        \frac{\eps k_i}{\sqrt{n}}& \text{w.p.} \,\, 1-p_i  \\
        \frac{\eps(k_i+1)}{\sqrt{n}}& \text{w.p.} \,\, p_i.
    \end{cases}
\end{equation}

The next result is a bit similar to Lemma 3.2 from \cite{L21}.

    \begin{prop}\label{prop:3}
        Let $m := \ceil{\delta n} + \ceil{n\rho^2/\eps^2 + 4n\rho/\eps}$, with $\delta,\rho,\eps \in (0,1)$. Then there exists a set $\cN \subset \RR^n$ of size $(C/(\delta^{3/2}\eps)^m$
        such that for any deterministic $N\times n$ matrix $A,$
        
        \begin{equation}\label{eq:approx}
        \sup_{y \in \comp(\delta,\rho)}\inf_{z \in \cN}\norm{A(y-z)} \leq \frac{2\eps}{\sqrt{n}}\hnorm{A}_{HS},
        \end{equation}
        
        where $C > 0$ is an absolute constant.
    \end{prop}
    \begin{proof}
    Let $m := \ceil{\delta n} + \ceil{n\rho^2/\eps^2 + 4n\rho/\eps}$. 
    We define $\cN$ as
    \[
    \cN := \lrpar{\frac{\eps}{\sqrt{n}}\ZZ \times \frac{\eps}{\sqrt{n}}\ZZ \times \cdots \times \frac{\eps}{\sqrt{n}}\ZZ} \cap {\sparse}(m) \cap (1+\eps)B_2^n.
    \]
    Let $y \in \comp(\delta,\rho)$ and let $\eta_y$ be the random rounding of $y$ with parameter $\eps$. On one hand, since $\E(y_i-\eta_i)^2\leq \frac{\epsilon^2}{n},$ we get
     \[
    \E\norm{A(y-\eta_y)}^2=\sum_{i=1}^n \E\langle  A^Te_i, y-\eta_y\rangle^2 \le\sum_{i=1}^n \frac{\eps^2}{n} |A^T e_i|^2=  \frac{\eps^2}{n}\hnorm{A}_{HS}^2.
    \]
   Therefore, by Markov's inequality,
    \begin{equation}\label{rounding-bnd}
    \bp\left(\norm{A(y - \eta_y)} > \frac{2\eps}{\sqrt{n}}\hnorm{A}_{HS}\right) < 1/2.
    \end{equation}
    By construction, $\eta_y\in \lrpar{\frac{\eps}{\sqrt{n}}\ZZ \times \frac{\eps}{\sqrt{n}}\ZZ \times \cdots \times \frac{\eps}{\sqrt{n}}\ZZ} \cap (1+\eps)B_2^n.$ It remains to note that 
    \begin{equation}\label{sparse}
    \bp(\eta_y\in {\sparse}(m))>0.75.
    \end{equation}
    Then combining (\ref{rounding-bnd}) and (\ref{sparse}) we conclude that with probability at least $0.25$, the vector $\eta_y\in \cN$ and simultaneously satisfies $|A(y-\eta_y)|\leq\frac{2\eps}{\sqrt{n}}\|A\|_{HS},$ and therefore there exists a realization of $\eta_y$ which satisfies the conclusion of our proposition. 

    Indeed, let us verify (\ref{sparse}). Without loss of generality we may assume that $y$ has non-negative entries satisfying $y_1 \ge y_2 \ge \cdots \ge y_n$. We define 

\[
    J := \{1,2,\cdots,\ceil{\delta n}\} \cup \{i \in [n]: y_i \ge \eps/\sqrt{n}\}.
\]

Since $\sum_{i > \ceil{\delta n}} y_i^2 \le \rho^2$ we may bound $|J|$ by

\[
    |J| \le \ceil{\delta n} + \frac{n\rho^2}{\eps^2}.
\]

    Since $\sum_{i \in J^c} y_i^2 \le \rho^2$, the definition of random rounding (e.g. see \eqref{eq:round-def}) gives
    
$$
    \E[\#\{i \in J^c:(\eta_y)_i=\frac{\eps}{\sqrt{n}}\}] = \sum_{i \in J^c}p_i\leq \frac{\sqrt{n}}{\eps}\cdot\sum_{i \in J^c} |y_i| \le \frac{n\rho}{\eps},
$$

where the last inequality follows from Cauchy-Schwarz. Therefore by Markov's inequality,

$$
\bp\left(\eta_y\not\in \sparse(m)\right) \le \bp\left(\#\{i \in J^c:(\eta_y)_i=\frac{\eps}{\sqrt{n}}\} > \frac{4n\rho}{\eps}\right)< \frac{(n\rho/\eps)}{(4n\rho/\eps)} = 0.25.
$$

This verifies \eqref{sparse}.

\medskip

We now bound $\#\cN$  using the following observations:
    
    \begin{enumerate}
        \item Each point in $\cN$ is contained in an $m$-dimensional coordinate hyperplane.
        
        \item On every $m$-dimensional coordinate hyperplane, each point in $\cN$ is the vertex of an $m$-dimensional hypercube of volume $(\eps/\sqrt{n})^m$ with the interiors of the hypercubes being disjoint and contained in $(1+\eps)B_2^m$.
        
        \item Each hypercube has $2^m$ vertices.
    \end{enumerate}
    
    Observe now that the number of $m$-dimensional coordinate hyperplanes is $\binom{n}{m}$ and, for a given hyperplane, the maximum number of hypercubes that can fit in $(1+\eps)B_2^m$ is at most $(1+\eps)^m{\rm vol}(B_2^m)/(\eps/\sqrt{n})^m$. Therefore we can bound $\#\cN$ by
    
    \begin{align*}
        \binom{n}{m} \cdot \frac{(1+\eps)^m{\rm vol}(B_2^m)}{(\eps/\sqrt{n})^m} \cdot 2^m 
        &\leq \lrpar{\frac{ne}{m}}^m \cdot C^m\lrpar{\frac{1}{\eps} + 1}^m \lrpar{\frac{n}{m}}^{m/2} \cdot 2^m,\\
        &\le C^m \cdot \delta^{-m} \cdot \eps^{-m} \cdot \delta^{-m/2} = \lrpar{\frac{C}{\delta^{3/2}\eps}}^m,
    \end{align*}
where in the second inequality we used the fact that $m \ge \delta n$.
\end{proof}

\begin{proof}[Proof of Theorem \ref{thm:comp}]
We take $\eps \in (0,1/2]$ to be a positive parameter whose value we will specify later. Let $m := \ceil{\delta n} + \ceil{n\rho^2/\eps^2 + 4n\rho/\eps}$. By Proposition \ref{prop:3} there exists a deterministic net, $\cN$, of size at most $(C_1/(\delta^{3/2}\eps))^m$, such that  
$$\sup_{y \in \comp(\delta,\rho)}\inf_{z \in \cN}\norm{A(y-z)} \leq \frac{(\eps/C_2)}{\sqrt{n}}\hnorm{A}_{HS}.$$  

Define the following events:
\[
\cE_1 := \{\hnorm{A}_{HS} \le C_2\sqrt{Nn}\}, \,\, \cE_2 := \{\inf_{y \in \cN} \norm{Ay} \ge 2\eps\sqrt{N}\}.
\]
When $\cE_1$ and $\cE_2$ are both true, we have that 
\[
\inf_{x \in \comp(\delta,\rho)} \norm{Ax} \ge \inf_{y \in \cN} \norm{Ay} - \sup_{y \in \comp(\delta,\rho)}\inf_{z \in \cN}\norm{A(y-z)} \ge 2\eps\sqrt{N} - \eps\sqrt{N} \ge \eps\sqrt{N}. 
\]
Thus to prove Theorem \ref{thm:comp}, it suffices to show that $\bp( \cE_1^c)+\bp(\cE_2^c) \le \exp(-c\sqrt{Nn})$. We first bound $\bp(\cE_1^c)$. By Paouris' inequality Theorem \ref{thm:Paouris}, we have 

\[
\bp(\cE_1^c) = \bp(\hnorm{A}_{HS} \ge C_2\sqrt{Nn}) \le \exp(-c\sqrt{Nn}).
\]

Next we bound $\bp(\cE_2^c)$. For a fixed unit vector $v$, the map $A \mapsto Av$ is a linear transformation from $\R^{N \times n}$ to $\R^n$. If $A$ is an isotropic matrix then $\E[Av] = 0$ and $\E[(Av)(Av)^\top]= I_N$, and so $Av$ is an isotropic log-concave random vector in $\R^N$. Then \eqref{eq:Bizeul} of Theorem \ref{thm:Bizeul} gives
\begin{equation*}\label{eq:point-wise:1}
\sup_{v \in \bS^{n-1}} \bp(\norm{Av} \le 4\eps\sqrt{N}) \le (C_3\eps)^N.
\end{equation*}
Since $\norm{Ay} = \norm{y}\norm{Ay/\norm{y}}$ and $\inf_{y \in \cN} \norm{y} \ge 1-\eps \ge 1/2$, the previous line implies that
\begin{equation*}\label{eq:point-wise:2}
\sup_{y \in \cN} \bp(\norm{Ay} \le 2\eps\sqrt{N}) \le (C_3\eps)^N.
\end{equation*}

Taking a union bound over $\cN$ and applying the point-wise bound for each net-point we get 
\[
\bp(\cE_2^c) \le \#\cN \cdot (C_3\eps)^N \le \lrpar{\frac{C_1}{\delta^{3/2}\eps}}^m (C_3\eps)^N.
\]
Taking $\rho = (1-\delta)\eps/5$ we get 
$$m \le \ceil{\delta n} + \ceil{2(1-\delta)n/5} < (1+\delta)n/2 \le (1+\delta)N/2.$$ Therefore
\[
\lrpar{\frac{C_1}{\delta^{3/2}\eps}}^m (C_3\eps)^N \le \lrpar{\frac{C^{1/(1-\delta)}\eps}{\delta^{(3/2)(1+\delta)/(1-\delta)}}}^{N(1-\delta)/2}.
\]
Taking $\eps =  \min(1/2,\delta^{(3/2)(1+\delta)/(1-\delta)}/(2C^{1/(1-\delta)}))$ we conclude that $\bp(\cE_2^c) \le e^{-cN}$.
Therefore, 
$$\bp(\cE_1^c) + \bp(\cE_2^c) \le e^{-c\sqrt{Nn}} + e^{-cN} \le e^{-c\sqrt{Nn}},$$ and the result follows. 
\end{proof}
\begin{rem}\label{rem:tall}
    We note that the proof of \refT{thm:comp} can be easily modified to give a proof of \refT{thm:tall} in the case of $N/n \in (1+\lambda,2)$. This is done by taking a net over the unit sphere, instead of just the compressible vectors, using \refL{lem:livshyts}. Then the constant $c$ in \eqref{eq:comp} of \refT{thm:comp} goes to 0 as $\lambda$ goes to 0. The point is that in this regime $\sqrt{Nn}$ is of order $N$, making the reduction to the compressible case unnecessary.
\end{rem}

\section{Unconditional square case}\label{sec:uncond}

In this subsection we will prove the small ball estimate for the smallest singular value of an unconditional isotropic log-concave square matrix. For the remainder of the subsection, we will take $\delta,\rho \in (0,1)$ to be absolute constants as guaranteed by Theorem \ref{thm:comp}. Since $\bS^{n-1} = \comp(\delta,\rho) \cup \incomp(\delta,\rho)$, we have
 
\[
\bp\lrpar{\inf_{x \in \bS^{n-1}} \norm{Ax} < \frac{\eps}{\sqrt{n}}} \leq \bp\lrpar{\inf_{x \in \comp(\delta,\rho)} \norm{Ax} < \frac{\eps}{\sqrt{n}}} + \bp\lrpar{\inf_{x \in \incomp(\delta,\rho)} \norm{Ax} < \frac{\eps}{\sqrt{n}}}.
\]

By \eqref{eq:comp} of Theorem \ref{thm:comp}, the first term of the right hand side is at most $e^{-cn}$.
To deal with the second term, we employ the Rudelson-Vershynin invertibility-via-distance approach for square matrices: see e.g. Lemma 3.5 of \cite{RV08} which yields
\[
\bp\lrpar{\inf_{x \in \incomp(\delta,\rho)} \norm{Ax} < \frac{\eps}{\sqrt{n}}} \leq \frac{4}{\delta n}\sum_{i = 1}^n \bp\lrpar{{\dist}(A_i,H_i) < \eps},
\]
where $H_i$ is the subspace spanned by all but the $i$th column of $A.$ Without loss of generality it suffices to show that $\bp\lrpar{{\dist}(A_1,H_1) < \eps} \leq C\eps + e^{-cn}$.

To this end, we first recall that incompressible vectors are spread, as was shown by Rudelson and Vershynin:
\begin{prop}\citep[Lemma 2.5]{RV08}\label{prop:spread}
    Let $x \in \incomp(\delta,\rho)$. Then there exists a set $J \subset [n]$ of size at least $\rho^2\delta n/2$ such that 
    \[
    \frac{\rho}{\sqrt{2n}} \leq |x_i| \leq \frac{1}{\sqrt{\delta n}},
    \]
    for all $i \in J$.
\end{prop}

Next we remark that $H_i$ is unlikely to intersect the set of compressible vectors:
\begin{prop}\label{prop:normal}
Let $A \in \R^{n \times (n-1)}$ be an isotropic log-concave matrix and let $H$ be the column span of $A$. Then $\bp(H^\perp \cap \comp(\delta,\rho) \neq \emptyset) \le e^{-cn}$.
\end{prop}
\begin{proof}
    If $x \in H^\perp$ then $A^\top x = 0$. Since $A^\top \in \R^{n \times (n-1)}$ is an isotropic log-concave matrix we have
    \[
    \bp \lrpar{H^\perp \cap \comp(\delta,\rho) = \emptyset} \le \bp\lrpar{\inf_{x \in \comp(\delta,\rho)} \norm{A^\top x} \le c\sqrt{n}} \le e^{-cn},
    \]
    with the last inequality following from \eqref{eq:comp} of Theorem \ref{thm:comp} applied with $A^{\top}$ in place of $A$. 
\end{proof}

Next, we recall that projections of isotropic log-concave vectors are isotropic and log-concave:
\begin{prop}\label{prop:proj}
    Let $X \in \R^n$ be an isotropic log-concave random vector and let $X_I \in \R^{I}$ denote the orthogonal projection of $X$ onto the coordinate subspace indexed by $I$. Then $X_I$ is an isotropic log-concave random vector.  
\end{prop}
\begin{proof} The fact that projections of log-concave vectors are log-concave follows from Prekopa-Leindler inequality (see e.g. Theorem 3.6 in \cite{Liv-notes}). Isotropicity follows from definition: indeed, a projection of a centered vector remains centered, and the covariance matrix of $X_I$ is a sub-matrix of the covariance matrix of $X$, and all sub-matrices of an identity matrix are identity matrices.
\end{proof}

Next, we prove lower tail estimates for order-statistics of isotropic log-concave random vectors using the slicing estimate of Klartag, Lehec \cite{KL24}, Bizeul \cite{B25}.

\begin{prop}[isotropic log-concave vectors are spread]\label{prop:small-ball}
    Let $X \in \R^n$ be an isotropic log-concave random vector. Let $X^*$ be the decreasing rearrangement of $X$ according to absolute value (i.e. $|X^*_1| \geq |X^*_2| \geq \cdots \geq |X^*_n|$). Consider constants $c_1 \in (0,1),$ $c_2 > 0$. Then there exists $r> 0$ which depends only on $c_1, c_2$ such that 
    \[
    \bp(|X^*_{n(1-c_1)}| \leq r) \leq e^{-c_2n}.
    \]
\end{prop}
\begin{proof}
Note the following inclusion of random events: 
\[
\lrcpar{|X^*_{n(1-c_1)}| < r} \subset \bigcup_{\overset{I \subset [n]}{\# I = c_1n}} \bigcap_{j \in I} \{|X_j| < r \}.
\]
 For a fixed $I \subset [n]$ with $\# I = c_1n$, the vector $X_I \in \R^I$ is isotropic and log-concave by Proposition \ref{prop:proj}. Therefore
\[
\bp\lrpar{\cap_{j \in I} \{|X_j| < r\}} \leq \bp\lrpar{\sum_{j \in I} X_j^2 \leq r^2c_1n} = 
\bp\lrpar{\norm{X_I}^2 < r^2c_1n} \leq (Cr)^{c_1n},
\]
with the last inequality following from the slicing bound of Klartag-Lehec \cite{KL24} and Bizeul \cite{B25}, see Theorem \ref{thm:Bizeul} in the preliminaries. Applying a union bound over all possible choices of $I$ we conclude that 
\[
\bp\lrpar{|X_{n-m}^*| < r }\leq \bp\left(\bigcup_{\overset{I \subset [n]}{\# I = c_1n}} \bigcap_{j \in I} \{|X_j| < r \}\right)\leq \binom{n}{c_1n}(Cr)^{c_1n} \leq \lrpar{\frac{Cre}{c_1}}^{c_1n}.
\]
Note now that we may take $r = (c_1/(Ce))e^{-c_2/c_1}$ to obtain the desired bound.
\end{proof}

\emph{Proof of Theorem \ref{thm:unconditional}.} As explained above, it suffices to bound $\bp(\dist(A_1,H_1) < \eps)$. Let $r$ be the absolute constant guaranteed by Proposition \ref{prop:small-ball} with $c_1 : =\delta^2\rho/8$ and $c_2 := 1$. 
Let $B \in \R^{n \times (n-1)}$ denote the sub-matrix of $A$ excluding the column $A_1$. Since $A$ has a density and $A_1,B$ partition the entries of $A$, the pair $(A_1,B)$ has a joint density. Therefore the conditional density of $A_1$ given $B$ is well-defined. Let $X^*$ denote the decreasing rearrangement of $X$ (i.e. $|X^*_1| \ge |X^*_2| \ge \cdots \ge |X^*_n|)$. We now define the following sets:
\begin{align*}
    \cS &:= \{X \in \R^n~:~|X^*_{n(1-\delta^2\rho/8)}| \ge r\}, \\
    \cT_1 &:= \{Y \in \R^{n \times (n-1)}: \bp(A_1 \not \in \cS| B = Y)\le e^{-n/2}\}, \\
    \cT_2 &:= \{Y \in \R^{n \times (n-1)}: \colspan (Y)^\perp \cap \bS^{n-1} \subseteq \incomp(\delta,\rho)\}, \\
    \cT_3 &:= \{Y \in \R^{n \times (n-1)}: \dim(\colspan (Y)^\perp) = 1\}. 
\end{align*}
Here, for a matrix $M,$ the notation $\colspan(M)$ stands for the span of the columns of $M.$

\medskip

We first bound $\bp(B \not \in\cT_2)$. By Proposition \ref{prop:normal} we have 
$$\bp(B \not \in\cT_2) \le \bp(H_1^\perp \cap \bS^{n-1} \not \subseteq \incomp(\delta,\rho)) \le e^{-cn}.$$ 
Next, we bound $\bp(B \not \in \cT_3)$. Note that $\dim(H_1^\perp) = 1$ if the columns of $A$ are linearly independent, which is equivalent to $\det(A) \neq 0$. 
Since $\det(A)$ is a polynomial in the entries of $A$ the corresponding level set can be viewed as a hypersurface (zero-set of a polynomial) in $\R^{n^2}$. Since $A$ follows an isotropic log-concave distribution, it is absolutely continuous with respect to Lebesgue measure. Since the Lebesgue measure of any hypersurface is 0 it follows that $\bp(\det(A) = 0) = 0$. Therefore, 
$$\bp(B \not \in \cT_3 )\le \bp(\dim(H_1^\perp) \neq 1) = 0.$$

Next, we bound $\bp(B \not \in \cT_1)$. By our choice of $r,c_2$, the definitions of $\cS$ and $\cT_1$, Proposition \ref{prop:small-ball} and Markov's inequality, we have $\bp(B \not \in \cT_1) \le e^{-n/2}$.
We conclude that
\[
\bp(\dist(A_1,H_1) < \eps) \le e^{-cn} + \bp(\dist(A_1,H_1) < \eps~|~ B \in \cT_1 \cap \cT_2 \cap \cT_3).
\]
Writing $Z := (A_1|B = Y)$, where $Y \in \cT_1 \cap \cT_2 \cap \cT_3$ is arbitrary, $Z$ is distributed as an unconditional log-concave random vector with diagonal covariance matrix. We will now argue that most of the diagonal entries of ${\rm Cov}(Z)$ are lower bounded by an absolute constant. Since $Y \in \cT_1$ we have that
\[
\sum_{i = 1}^n \min(1,\E(Z_i^2)/r^2) \ge \bp(|Z_{n(1-\delta^2\rho/8)}^*| \ge r)\cdot n(1-\delta^2\rho/8) \ge n(1-\delta^2\rho/6),
\]
where in the last inequality we used the fact that $(1-e^{-n/2})(1-\delta^2\rho/8) \ge (1-\delta^2\rho/6)$.
A further application of Markov's inequality gives
\[
\#\{i \in [n]: \E(Z_i^2)/r^2 \le 1/2\} \le n\delta^2\rho/3.
\]
In particular $\E(Z_i^2) \ge r^2/2$ for at least $n(1-\delta^2\rho/3)$ indices. Now, let $\eta \in \colspan(B)^\perp \cap \bS^{n-1}$ and consider the following sets of indices:

\begin{align*}
\cI &:= \{j : \E[Z_j^2] \ge r^2/2 \},\\
\cJ &:= \{j : |\eta_j| \ge \rho/(\sqrt{2n})\}.
\end{align*}

By our previous calculation, Proposition \ref{prop:spread} and the fact that $B \in \cT_2$ we have
\[
|\cI \cap \cJ| \ge |\cI| + |\cJ| - |\cI \cup \cJ| \ge n(1-\delta^2\rho/3) + n\delta^2\rho/2 - n = n\delta^2\rho/6.
\]
Since ${\rm Cov}(Z)$ is diagonal it follows that
\[
\E[|\langle Z, \eta \rangle|^2] = \sum_{i=1}^n \E[Z_i^2]\eta_j^2 \geq \frac{n\rho^2\delta}{6}\cdot \frac{r^2}{2} \cdot \frac{\rho^2}{2n} = \frac{r^2\rho^4\delta}{24}.
\]
In particular $\E[|\langle Z, \eta\rangle|^2]$ is bounded below by an absolute constant. Therefore, $\langle Z, \eta \rangle$ is a log-concave random variable that has mean 0 and non-zero constant variance. In particular it has a bounded density and satisfies $\bp(|\langle Z,\eta\rangle| < \eps) < C\eps$.  Now, since $B \in \cT_3$, we know that $\dist(Z,\colspan(B)) = |
\langle Z,\eta \rangle|$. Since the choice of $Y \in \cT_1 \cap \cT_2 \cap \cT_3$ was arbitrary we conclude that 
\[
\bp(\dist(A_1,H_1) < \eps|B \in \cT_1 \cap \cT_2 \cap \cT_3) \le C\eps.
\]
We conclude that 
$$\bp\left(\dist(A_1,H_1) < \eps\right) \le C\eps + e^{-cn}.\, \square$$

\begin{rem}
    In conjunction with upper tail estimates for order statistics of isotropic log-concave (see \citep[Proposition 1]{L10}), one can deduce the following: Let $(X,Y) \in \R^{n^2}$ be an unconditional isotropic log-concave random vector in $\R^{n^2}$, with $X \in \R^{n}$ being the first $n$ coordinates. Then, with exponentially small failure probability, $X$, conditioned on $Y$, is unconditional log-concave and, up to permutation of its coordinates, has a diagonal covariance matrix $\Sigma$ satisfying $r \leq \Sigma_{ii} \leq R$ for $1 \leq i \leq n(1-c_1)$ and $rnc_1 \leq \sum_{n(1-c_1) < i \leq n} \Sigma_{ii} \leq Rnc_1$. Here $r,R > 0$ are constants that depend on $c_1$ (as $c_1 \to 0$ we have  $r \to 0$ and $R \to \infty$).
\end{rem}

\begin{rem}
    In the proof of \refT{thm:unconditional}, the unconditionality of the matrix is crucial to our proof strategy. Indeed, our proof relies on the fact that $\langle Z,\eta \rangle$ is distributed as a log-concave random variable. The unconditional assumption allows us to deduce that $\langle Z,\eta \rangle$ has mean 0 and constant variance, thus satisfying a small ball estimate at every scale. Our approach for deriving a lower bound on $\E|\langle Z,\eta \rangle|^2$ hinges on the fact that ${\rm Cov}(Z)$ has a diagonal covariance matrix which reduces the problem to lower bounding the size of most entries of $Z$. Without unconditionality, we must instead lower bound $\E|\langle Z - \E Z,\eta \rangle|^2$. In this setting, we cannot use \refPp{prop:small-ball} to lower bound the contribution from diagonal terms to $\E|\langle Z - \E Z,\eta \rangle|^2$, since a-priori entries of $Z$ can be large but $Z - \E Z$ close to 0. Furthermore, $\sum_{ij} {\rm Cov}(Z - \E Z)_{ij}$ is not sufficiently concentrated so as to prevent non-trivial contribution from off-diagonal terms to $\E|\langle Z - \E Z,\eta \rangle |^2$. This is because $n^{-1}\sum_{ij} {\rm Cov}(A_1)_{ij} = \langle n^{-1/2}A_1,(1,1,\cdots,1)\rangle^2$ is distributed as the square of a mean 0 variance 1 log-concave random variable.
\end{rem}

\section{Independent columns case and the general aspect ratio}

In this subsection, we will prove the small ball type bounds on the smallest singular value for rectangular log-concave matrices with independent and isotropic columns. Our scheme relies on the ideas from the work of Rudelson, Vershynin \cite{RudVer-general}. We take $\delta,\rho \in (0,1)$ to be absolute constants as guaranteed by \refT{thm:comp}. In view of Theorem \ref{thm:comp} and Lemma \ref{lem:invert-via-dist}, we will use different strategies that change when $N/n$ crosses a certain threshold. We have two cases, depending on the aspect ratio of the matrix.

\subsection{Aspect ratio: $N/n \ge 1+\lambda$}
In this regime, Theorem \ref{thm:tall} is applicable. Therefore we have that $$\bp(\inf_{x \in \bS^{n-1}} \norm{Ax} \le c\sqrt{N}) \le e^{-cN},$$ where $c >0$ depends only on $\lambda$. Taking $\lambda > 0$ to be a sufficiently small constant, it suffices to prove \refT{thm:indep-columns} in the following regime of aspect ratio.
\subsection{Aspect ratio: $1 \le N/n \le 1+\lambda$}
By Theorem \ref{thm:comp}, we have that $$\bp\lrpar{\inf_{x \in \comp(\delta,\rho)}\norm{Ax} \le c\sqrt{N}} \le e^{-cN},$$ 
for some absolute constant $c > 0$. We are thus left with lower bounding $\norm{Ax}$ for $x \in \incomp(\delta,\rho)$. 
Our proof is based on the invertibility-via-distance approach. We begin by recording some technical lemmas related to this. The first lemma will require a definition.
 \begin{definition}\label{def:spread}
     A vector $v \in \bS^{d-1}$ is said to be spread if all entries satisfy $|v_i| \in [c/\sqrt{d},C/\sqrt{d}]$, for some positive constants $c,C$. We denote the set of spread vectors as ${\rm spread}_d$.
\end{definition}
\begin{lem}\cite[Lemma 6.2]{RV08}\label{lem:invert-via-dist}
    Let $d \in \{1,2\cdots,n\}$ and $A \in \R^{N \times n}$ be a random matrix. Then there exists $J \subset \{1,2,\cdots,n\}$ of size $d$ such that 
    \[
    \bp\lrpar{\inf_{x \in Incomp(\delta,\rho)}\norm{Ax} \leq \frac{\eps d}{\sqrt{n}}} \leq
    C^d\bp\lrpar{\inf_{x \in {\rm spread}_d} \dist(A_Jx,H_{J^c}) \leq \eps\sqrt{d}},
    \]
    where $C > 0$ and the spread parameters in Definition \ref{def:spread} depend only on $\delta$ and $\rho$.
\end{lem}
For the rest of this section we will take $d:= N-n+1$. Note that \refL{lem:invert-via-dist} is applicable for this choice of $d$ and regime of aspect ratio since $N/n < 2$ implies that $N-n < n,$ hence $d = N-n+1 \le n$. For the rest of this subsection we will take $J$ to be set of indices guaranteed by Lemma \ref{lem:invert-via-dist}. For convenience we write $W := P_{H_{J^c}^\perp}A_J$. Note that $\dist(A_Jx,H_{J^c}) = \norm{Wx}$. We now give a small ball estimate for $\norm{Wx}$.
\begin{prop}\label{prop:net-bound-1}
Let $w \in \R^n$ and $x \in \frac{3}{2}B_2^n \setminus \frac{1}{2}B_2^n$. Then
\begin{equation}\label{eq:net-bound-1}
     \bp\lrpar{\norm{Wx - w} \leq t\sqrt{2d-1}} \leq (Ct)^{2d-1},
\end{equation}
where $C > 0$ is an absolute constant.
\end{prop}
\begin{proof}
    Note that the columns of $A$ are linearly independent with probability 1 (we argued this in our proof of Theorem \ref{thm:unconditional} in section \ref{sec:uncond}). Therefore we may assume that $H_{J^c}$ has co-dimension $N - |J^c| = N-(n-d) = 2d-1$. From the definition of $A$ and the independence of its columns, the columns of $A_J$,  conditioned on $H_{J^c}$, are distributed as independent isotropic log-concave vectors. Therefore $A_Jx/\norm{x}$ is distributed as an $N$-dimensional isotropic log-concave vector and $Wx/\norm{x}$ is distributed as a $2d-1$-dimensional isotropic log-concave vector in $H_{J^c}^\perp$. Observe now that
    \[
    \norm{Wx - w} \ge \norm{Wx - P_{H_{J^c}^\perp}w} \ge (1/2)\norm{Wx/\norm{x} - P_{H_{J^c}^\perp }w/\norm{x}}.
    \]
    Since $P_{H_{J^c}^\perp}w \in H_{J^c}^\perp$ is fixed when conditioned on $H_{J^c}$ we have, by Theorem \ref{thm:Bizeul}, that 
    \[
    \bp(\norm{Wx/\norm{x} - P_{H_J^c}w/\norm{x}} \le 2t\sqrt{2d-1}) \le (Ct)^{2d-1}.
    \]
    The result follows.
\end{proof}
 Lastly, because $d$ may be much smaller than $N$, we will need an additional ingredient. In particular, we will need the following decoupling result. 
\begin{lem}\citep[Lemma 8.8]{L21}
\label{lem:decouple}
    Let $d \geq 2$. Let $W$ be an $N \times d$ random matrix with independent columns. Let $z \in \frac{3}{2}B_2^d \setminus \frac{1}{2}B_2^d$ be such that $\abs{z_k} \geq c/\sqrt{d}$. Then for every $0 < a < b$ we have 
    \[
    \bp\lrpar{\norm{Wz} < a, \hnorm{W}_{HS} > b} \leq 2 \bp\lrpar{\hnorm{W}_{HS} > \frac{b}{2}} \sup_{x \in \frac{3}{2}B_2^d \setminus \frac{1}{2}B_2^d, w \in \R^N} \bp\lrpar{\norm{Wx - w} < Ca},
    \]
    where $C > 0$ is an absolute constant.
\end{lem}
Note that Lemma \ref{lem:decouple} is applicable for $z \in {\rm spread}_d$. 
\begin{proof}[Proof of \refT{thm:indep-columns}]
As argued earlier, it suffices to show that

\[
\bp\lrpar{\inf_{x \in \incomp(\delta,\rho)}\norm{Ax} \le \frac{\eps d}{\sqrt{n}}} \le (C\eps)^d + e^{-cN}
\]

when $1 \le N/n \le 1+\lambda$ for $\lambda > 0$ a sufficiently small constant. Therefore, by \refL{lem:invert-via-dist} it suffices to show that 

\[
\bp\lrpar{\inf_{x \in {\rm spread}_d} \norm{Wx} \le \eps\sqrt{d}} \le (C\eps)^d + e^{-cN},
\]
since, in view of $N \le (1+\lambda)n$, we may take $\lambda$ sufficiently small so that $C^de^{-cN} \le e^{-(c/2)N}$.
To that end we define the following sets of events:
\begin{align*}
    \cE_1 &:= \left\{ \inf_{x \in {\rm spread}_d}\norm{Wx} \le \eps\sqrt{d}, \hnorm{W}_{HS} < C_1d\right\}. \\
    \cE_{2,i} &:= \left\{\inf_{x \in {\rm spread}_d}\norm{Wx} \leq 2^i\eps\sqrt{d}, C_12^id \le \hnorm{W}_{HS} < C_12^{i+1}d \right\}.\\
    \cE_3 &:= \left\{\hnorm{W}_{HS} \geq C_1\sqrt{Nn}\right\}.
\end{align*}

Taking $m = \lceil \log_2(C_1\sqrt{Nn}/d)\rceil$,
one can check that the event $\lrset{\inf_{x \in {\rm spread}_d} \norm{Wx} \leq \eps\sqrt{d}}$ is contained in $\cE_1 \cup \lrpar{\cup_{1 \leq i \leq m} \cE_{2,i}} \cup \cE_3$.
Therefore,

\begin{align*}
\bp\lrpar{\inf_{x \in {\rm spread}_d}\norm{Wx} \leq \eps\sqrt{d}} \leq 
\bp\lrpar{\cE_1 \cup \lrpar{\cup_{1 \leq i \leq m}\cE_{2,i}}  \cup \cE_3} \leq \bp(\cE_1) + \bp(\cE_3) + \sum_{i = 0}^m \bp(\cE_{2,i}).
\end{align*}

We now bound these probabilities. We start with $\bp(\cE_3)$. Note that the entries of $W$ are jointly distributed as a $d^2$-dimensional isotropic log-concave vector in $\R^{Nd}$. Therefore by Paouris's inequality we have  

\[
\bp\lrpar{\cE_3} \leq \exp(-c\sqrt{Nn}) \le \exp(-c'N),
\]
where for the last inequality we used the assumption on the aspect ratio.  

We next bound $\bp(\cE_1)$.
    Let $\cN \subset \frac{3}{2}B_2^d \setminus \frac{1}{2}B_2^d$ be the net guaranteed by Lemma \ref{lem:net-full} of cardinality at most $(C_2/\eps)^{d-1}$ satisfying
    
    \[
    \sup_{x \in \bS^{n-1}} \inf_{y \in \cN} \norm{W(x-y)} \le \frac{(\eps/C_1)}{\sqrt{d}}\hnorm{W}_{HS}.
    \]
    
    Suppose now that $\inf_{y \in \cN} \norm{Wy} > 2\eps\sqrt{d}$ and $\hnorm{W}_{HS} < C_1d$. Then
    
    \[
    \inf_{x \in {\rm spread}_d}\norm{Wx} > \inf_{y \in \cN}\norm{Wy} - \sup_{x \in {\rm spread}_d} \inf_{y \in \cN} \norm{W(x-y)} \ge 2\eps\sqrt{d} - \eps\sqrt{d} = \eps \sqrt{d}.
    \]
    
    Therefore 
    
    \begin{align*}
    \bp\lrpar{\cE_1} 
    &\leq \bp\lrpar{\inf_{y \in \cN} \norm{Wy} \leq 2\eps\sqrt{d},\hnorm{W}_{HS} < C_1d}, \\
     &\leq \bp\lrpar{\inf_{y \in \cN} \norm{Wy} \leq 2\eps\sqrt{d}},\\
    &\leq |\cN|\sup_{x \in \frac{3}{2}B_2^d\setminus \frac{1}{2}B_2^d}\bp\lrpar{\norm{Wx} \leq 2\eps\sqrt{d}}\\
    &\leq (C_2/\eps)^{d-1}(C_3\eps)^{2d-1} \\
    &\leq (C\eps)^{d}.
    \end{align*}
    
We now bound $\bp(\cE_{2,i})$. Suppose that $\inf_{y \in \cN}\norm{Wy} > 2^{i+2}\eps\sqrt{d}$ and $\hnorm{W}_{HS} < C_12^{i+1}d$. Then

\[
    \inf_{x \in {\rm spread}_d}\norm{Wx} \ge \inf_{y \in \cN}\norm{Wy} - \sup_{x \in {\rm spread}_d} \inf_{y \in \cN} \norm{W(x-y)} \ge 2^{i+2}\eps\sqrt{d} - 2^{i+1}\eps\sqrt{d} = 2^{i+1}\eps \sqrt{d}.
\]    

    Therefore, using Lemma \ref{lem:decouple}, we have
    
    \begin{align*}
    \bp(\cE_{2,i}) 
    & \leq \bp\lrpar{\inf_{y \in \cN_i} \norm{Wy} \le 2^{i+2}\eps\sqrt{d}, C_12^id \le \hnorm{W}_{HS} < C_12^{i+1}d},\\
    & \leq \bp\lrpar{\inf_{y \in \cN_i} \norm{Wy} \leq 2^{i+2}\eps\sqrt{d}, \hnorm{W}_{HS} \ge C_12^id},\\
    &\leq \sum_{y \in \cN_i}\bp\lrpar{\norm{Wy} \leq 2^{i+2}\eps\sqrt{d}, \hnorm{W}_{HS} \geq C_12^id},\\
    &\leq 2|\cN_i|\bp\lrpar{\hnorm{W}_{HS} \geq C_12^id}\sup_{x \in \frac{3}{2}B_2^n\setminus\frac{1}{2}B_2^n, w \in \R^N}\bp\lrpar{\norm{Wx-w} < C_32^i\eps\sqrt{d}},\\
    &\leq 
    2(C_2/\eps)^{d-1}\exp(-c2^id)(C_42^i\eps)^{2d-1}, \\
    &\leq (C\eps)^d\exp((2d-1)i\log(2)-c2^id).
    \end{align*}
    
Since $\sum_{i \ge 1} \exp((2d-1)i\log(2)-c2^id) < C^d$, we conclude that $ \sum_{i \ge 1}\bp(\cE_{2,i}) \le (C\eps)^d$. All together, we have that

\[
\bp(\cE_1) + \sum_{i \ge 1}\bp(\cE_{2,i}) + \bp(\cE_3)  \leq (C\eps)^d + (C\eps)^d + e^{-cN} \le (C\eps)^d + e^{-cN} \le (C\eps)^{N-n+1} + e^{-cN},
\] 

with the last inequality following from the fact that  $d=N-n+1$.
\end{proof}

\end{document}